\documentclass{article}
\usepackage{amssymb}
\usepackage{graphicx}
\usepackage{psfrag}

\newtheorem{thm}{Theorem}[section]
\newtheorem{lem}{Lemma}[section]
\newtheorem{defn}{Definition}[section]
\newtheorem{prop}{Proposition}[section]
\newtheorem{cor}{Corollary}[section]

\newcommand{\mb}{\mathbf}

\newenvironment{proof}{\medskip \noindent
{\bf Proof.}}{\hfill \rule{.5em}{1em} \\}

\title{Equivalent metrics and compactifications}
\author{Young Deuk Kim \\ School of Mathematical Sciences\\
Seoul National University\\ Seoul, 151-747, Korea\\(ydkimus@yahoo.com)}
\date{\today}

\begin{document}
\maketitle

\begin{abstract}
Let $(X,d)$ be a metric space and $m\in X$. Suppose that $\phi:X\times
X\to\mb{R}$ is a nonnegative symmetric function. We define a metric
$d^{\phi,m}$ on $X$ which is equivalent to $d$.
If $d^{\phi,m}$ is totally bounded, its completion is a compactification of
$(X,d)$. As examples, we construct two compactifications of
$(\mb{R}^s,d_E)$, where $d_E$ is the Euclidean metric and $s\geq 2$.\\

\noindent
{\footnotesize{\textbf{key words.}}} equivalent metric; completion;
compactification
\end{abstract}

\section{The metric $d^{\phi,m}$}\label{sec;metric}

Let $(X,d)$ be a metric space and $m\in X$. Suppose that
$\phi:X\times X\to\mb{R}$ is a nonnegative symmetric function.
As usual, two metrics $d_1$ and $d_2$ on a set $X$ are called equivalent if
$(X,d_1)$ and $(X,d_2)$ are homeomorphic. In this section, we will define a
metric $d^{\phi,m}$ on $X$ which is equivalent to $d$. 

\indent
For each $x,y\in X$, let
$$\delta^{\phi,m}(x,y)=\min\left\{ d(x,y),\ {1\over{1+d(m,x)}}+\phi(x,y)+
 {1\over{1+d(m,y)}}\right\}.$$

\noindent
And for each $x,y\in X$ and $n\in\mb{N}$, let
$$\Gamma^n_{x,y}=\{\, (x_0,\cdots,x_n)\mid x_0=x,x_n=y\mbox{ and }
x_i\in X\mbox{ for all }i\, \}$$
and
$$\Gamma_{x,y}=\bigcup_{n\in\mb{N}}\Gamma^n_{x,y}.$$

\indent
Notice that $\Gamma_{x,y}\neq\emptyset$ for all $x,y\in X$. In the following
definition, the infimum runs over all elements of $\Gamma_{x,y}$.

\begin{defn}
Suppose that $x,y\in X$. Let
\begin{equation}\label{eq;metric}
d^{\phi,m}(x,y)=\inf_{\Gamma_{x,y}}\sum^n_{i=1}\delta^{\phi,m}(x_{i-1},x_i).
\end{equation}
\end{defn}

\noindent
For the sake of simplicity, we will simply write $d^\phi$, $\delta^\phi$
to denote $d^{\phi,m}$, $\delta^{\phi,m}$ respectively. In particular, we write
eq. (\ref{eq;metric}) as
$$d^\phi(x,y)=\inf_{\Gamma_{x,y}}\sum^n_{i=1}
\delta^\phi(x_{i-1},x_i).$$

\noindent
Notice that $(x,y)\in\Gamma_{x,y}$, and therefore
\begin{equation}\label{eq;trivial}
d^\phi(x,y)=\inf_{\Gamma_{x,y}}\sum^n_{i=1}\delta^\phi(x_{i-1},x_i)
\leq\delta^\phi(x,y)\leq d(x,y).
\end{equation}
Notice also that $d^\phi$ is nonnegative. Therefore from eq.
(\ref{eq;trivial}), we have
\begin{equation}\label{eq;same}
d^\phi(x,x)=0\quad\mbox{for all }x\in X.
\end{equation}

\indent
The following subset $\Delta_{x,y}$ of $\Gamma_{x,y}$ is useful in the proof
of Lemma \ref{lem;useful}.
$$\Delta_{x,y}=\{\, (x_0,\cdots,x_n)\in\Gamma_{x,y}\mid
\delta^\phi(x_{i-1},x_i)\neq d(x_{i-1},x_i)\ \mbox{for some }
1\leq i\leq n\, \}.$$

\begin{lem}\label{lem;useful}
Suppose that $d^\phi(x,y)\neq d(x,y)$. Then
$$d^\phi(x,y)\geq {1\over{2(1+d(m,x))}}.$$
\end{lem}

\begin{proof}
Suppose that $d^\phi(x,y)\neq d(x,y)$. By eq. (\ref{eq;same}) we have
$x\neq y$, and  by eq. (\ref{eq;trivial}) we have
\begin{equation}\label{eq;neq}
d^\phi(x,y)<d(x,y).
\end{equation}
If $(x_0,\cdots,x_n)\in\Gamma_{x,y}-\Delta_{x,y}$, then
$$\sum^n_{i=1}\delta^\phi(x_{i-1},x_i)=\sum^n_{i=1}
d(x_{i-1},x_i)\geq d(x,y).$$
Therefore from eq. (\ref{eq;neq}), we have $\Delta_{x,y}\neq\emptyset$ and
\begin{equation}
d^\phi(x,y)=\inf_{\Delta_{x,y}}\sum^n_{i=1}\delta^\phi(x_{i-1},x_i).
\label{eq;up}
\end{equation}

\indent
Suppose that $(x_0,\cdots,x_n)\in\Delta_{x,y}$.
Let $k$ be the smallest integer such that $\delta^\phi(x_k,x_{k+1})\neq
d(x_k,x_{k+1})$. Notice that if $k\geq 1$ then
$$\delta^\phi(x_{i-1},x_i)=d(x_{i-1},x_i)\quad\mbox{for all }
1\leq i\leq k.$$

\indent
If $d(x_0,x_k)\geq 1+d(m,x_0)$ then we have $k\geq 1$, and therefore
\begin{eqnarray}
\sum^n_{i=1} \delta^{\phi}(x_{i-1},x_i)&\geq& \sum^k_{i=1}
\delta^{\phi}(x_{i-1},x_i)\nonumber \\
&=&\sum^k_{i=1} d(x_{i-1},x_i) \nonumber\\
&\geq& d(x_0,x_k) \nonumber\\
&\geq& 1+d(m,x_0)\nonumber\\
&=& 1+d(m,x).\label{eq;lem1}
\end{eqnarray}

\indent
If $d(x_0,x_k)<1+d(m,x_0)$ then
$$1+d(m,x_k)\leq 1+d(m,x_0)+d(x_0,x_k)<2+2d(m,x_0).$$
Therefore
\begin{eqnarray}
\sum^n_{i=1}\delta^\phi(x_{i-1},x_i)
&\geq& \delta^\phi(x_k,x_{k+1})\nonumber\\
&=& {1\over {1+d(m,x_k)}}+\phi(x_k,x_{k+1})+{1\over {1+d(m,x_{k+1})}}
\nonumber\\
&>& {1\over {1+d(m,x_k)}} \nonumber\\
&>& {1\over{2(1+d(m,x_0))}} \nonumber\\
&=& {1\over{2(1+d(m,x))}}.\label{eq;lem2}
\end{eqnarray}

Hence from eq. (\ref{eq;up}), (\ref{eq;lem1}) and (\ref{eq;lem2}), we have
$$d^\phi(x,y)\geq \min\left\{1+d(m,x),\ {1\over{2(1+d(m,x))}}\right\}
={1\over{2(1+d(m,x))}}.$$
\end{proof}

\indent
Now we show that $d^\phi$ is a metric on $X$.

\begin{thm}\label{thm;metric}
$d^\phi$ is a metric on $X$.
\end{thm}

\begin{proof}
From eq. (\ref{eq;metric}) and (\ref{eq;same}), recall that $d^\phi$ is
nonnegative and $d^\phi(x,x)=0$ for all $x\in X$.
Suppose that $d^\phi(x,y)=0$. By Lemma \ref{lem;useful}, we have
$d(x,y)=d^\phi(x,y)=0$. Thus $x=y$.

\indent
Suppose that $x,y\in X$. Notice that
$(x_0,x_1,\cdots,x_n)\in\Gamma_{x,y}$ if and only if
$(x_n,x_{n-1},\cdots,x_0)\in\Gamma_{y,x}$.
Since $\phi$ is symmetric, so is $\delta^\phi$. Therefore
$$\sum^n_{i=1}\delta^\phi(x_{i-1},x_i)=\sum^n_{i=1}\delta^
\phi(x_{n+1-i},x_{n-i})
\quad\mbox{for all } (x_0,x_1,\cdots,x_n)\in\Gamma_{x,y}.$$
Hence $d^\phi(x,y)=d^\phi(y,x)$.

\indent
Suppose that $x,y,z\in X$ and $\epsilon>0$.
There exist $(x_0,x_1,\cdots,x_n)\in\Gamma_{x,y}$ and
$(y_0,y_1,\cdots,y_m)\in\Gamma_{y,z}$ such that
$$\sum^n_{i=1}\delta^\phi(x_{i-1},x_i)<d^\phi(x,y)+{\epsilon\over 2}\quad
\mbox{and}\quad
\sum^m_{j=1}\delta^{\phi}(y_{j-1},y_j)<d^\phi(y,z)+{\epsilon\over 2}.$$
Notice that $(x_0,\cdots,x_n=y=y_0,\cdots,y_m)\in\Gamma_{x,z}$.
Therefore
\begin{eqnarray*}
d^\phi(x,z)&\leq& \sum^n_{i=1}\delta^\phi(x_{i-1},x_i)+
\sum^m_{j=1}\delta^\phi(y_{j-1},y_j)\\
&<&d^\phi(x,y)+{\epsilon\over 2}+
d^\phi(y,z)+{\epsilon\over 2}\\
&=&d^\phi(x,y)+d^\phi(y,z)+\epsilon.
\end{eqnarray*}
Since $\epsilon$ is arbitrary, we have
$d^\phi(x,z)\leq d^\phi(x,y)+d^\phi(y,z)$.
\end{proof}

\indent
By the following lemma, the identity map from $(X,d^\phi)$ to $(X,d)$ is
continuous.

\begin{lem}\label{lem;iso}
For all $x\in X$, there exists an open ball $B_x$ in $(X,d^\phi)$, with
center $x$, such that $d^\phi(y,z)=d(y,z)$ for all $y,z\in B_x$.
\end{lem}

\begin{proof}
For each $x\in X$, let
$$B_x=\left\{y\in X\mid d^\phi(y,x)<{1\over 8(1+d(m,x))}\right\}.$$
Suppose that $y\in B_x$. By Lemma \ref{lem;useful}, we have
$d^\phi(x,y)=d(x,y)$, and therefore
\begin{eqnarray}
d(m,y)&\leq&d(m,x)+d(x,y)\nonumber \\
&=&d(m,x)+d^\phi(x,y)\nonumber \\
&<&d(m,x)+{1\over 8(1+d(m,x))}\nonumber \\
&<&d(m,x)+1+d(m,x)\nonumber \\
&=&1+2d(m,x).\label{eq;bx}
\end{eqnarray}

\indent
Suppose that $y,z\in B_x$. From eq. (\ref{eq;bx}), we have
$1+d(m,y)<2+2d(m,x)$. Therefore
\begin{eqnarray*}
d^\phi(y,z)&\leq& d^\phi(y,x)+d^\phi(x,z)\\
&<&{1\over 8(1+d(m,x))}+{1\over 8(1+d(m,x))}\\
&=&{1\over 4(1+d(m,x))}\\
&<&{1\over 2(1+d(m,y))}.
\end{eqnarray*}
Hence by Lemma \ref{lem;useful}, we have $d^\phi(y,z)=d(y,z)$.
\end{proof}

\indent
By the following corollary, $d^\phi$ is equivalent to $d$ for all $\phi$ and
$m$.
\begin{cor}\label{cor;homeo}
The identity map from $(X,d^\phi)$ to $(X,d)$ is a homeomorphism.
\end{cor}

\begin{proof}
By eq. (\ref{eq;trivial}) and Lemma \ref{lem;iso}, it is trivial.
\end{proof}

\section{The compactification}\label{sec;cpt}

A compactification of a topological space $X$ is a compact Hausdorff
space $Y$ containing $X$ as a subspace such that $\overline{X}=Y$.
It is known that every metric space has a compactification
(see \cite{MUNKRES}, \S 38).
With the equivalent metric in the previous section, we are able
to construct various compactifications of a metric space.

\indent
Let $(X,d)$ be a metric space. Suppose that $m\in X$ and
$\phi:X\times X\to\mb{R}$ is a nonnegative symmetric function.
To get a compactification, we assume that
$$(X,d^\phi)=(X,d^{\phi,m})\quad\mbox{is totally bounded},$$
ie. there is a finite covering by $\epsilon$ balls for every $\epsilon>0$.
Then our compactification of $(X,d)$ is the completion 
$(\overline{X},\rho)$ of the totally bounded metric space $(X,d^\phi)$.

\indent
Notice that $X$ is a dense subset of $\overline{X}$ and $(\overline{X},\rho)$
is a compact metric space (see \cite{MUNKRES}, \S 45 and \cite{DUGUNDJI},
\S XIV.3 for details).
$\overline{X}$ can be considered as the set of equivalence classes of all
Cauchy sequences in $(X,d^\phi)$ with the equivalence relation
(see {\cite{GAAL}, \S V.7})
$$x_i\sim y_i\quad \mbox{if and only if}\quad
  \lim_{i\to\infty}d^\phi(x_i,y_i)= 0,$$
where a point $x$ in $X$ is identified to the equivalence class of constant
Cauchy sequence $\{x\}$.

\indent
Suppose that $\{x_i\},\{y_i\}\in\overline{X}$.
The metric $\rho$ is given by
$$\rho(\{x_i\},\{y_i\})=\lim_{i\to\infty}d^\phi(x_i,y_i).$$
In particular, we have
$$\rho(\{x\},\{y\})=d^\phi(x,y)\quad\mbox{for all }x,y\in X.$$

\indent
In 2002, the author had tried to apply this compactification to the research on
the tameness conjecture of Marden(\cite{MARDEN}) which was proved by
Agol(\cite{AGOL}) and Calegari-Gabai(\cite{CG}) in 2004, independently.
The author think that the compactification could be useful in the study of
Teichm\"uller space. In the next two sections, we apply the compactification
to the Euclidean metric space $\mb{R}^s$ with $s\geq 2$.

\section{The standard compactification of $(\mb{R}^s,d_E)$}\label{sec;scpt}

Let $O=(0,\cdots,0)\in\mb{R}^s$. We write $d_E$ to denote the Euclidean metric
on $\mb{R}^s$. In this section, as an example of the compactification in
Section \ref{sec;cpt}, we construct a compactification of $(\mb{R}^s,d_E)$, 
which will be called {\em the standard compactification}, which is homeomorphic
to the Euclidean closed unit ball
$$B^s=\{x\in\mb{R}^s\mid d_E(O,x)\leq 1\}.$$
Notice that we need to define a nonnegative symmetric function
$\phi: \mb{R}^s\times\mb{R}^s\to\mb{R}$ such that
$$(\mb{R}^s,d^\phi)=(\mb{R}^s,d_E^{\phi,O})$$
is totally bounded, where we wrote $d^\phi$ to denote $d_E^{\phi,O}$ for the
sake of simplicity.

\indent
For all $m\in\mb{N}$, let
$$ a_m=1+\frac{1}{2}+\cdots+\frac{1}{m} $$
and
$$ S_m=\{ x\in\mb{R}^s\mid d_E(O,x)=a_m\}. $$
Note that $a_m$ is an increasing sequence and $\lim_{m\to\infty}a_m=\infty$.

\indent
For all $p,q\in\mb{N}$, let $h_{p,q}:S_p\to S_q$ be the homeomorphism
defined by
$$h_{p,q}(x)=\frac{a_q}{a_p}\, x \quad\mbox{for all }x\in S_p.$$
Notice that if $h_{p,q}(x)=y$ then $ h_{q,p}(y)=x$.
We define the nonnegative symmetric function $\phi$ as follows.
\begin{defn}
\begin{eqnarray*}
\phi(x,y)=\left\{
\begin{array}{cl}
0 &\mbox{ if }\ h_{p,q}(x)=y\ \mbox{for some }p,q\in\mb{N}\\
\frac{1}{a_m}\, d_E(x,y)=d_E\left(\frac{x}{a_m},\frac{y}{a_m}\right) 
&\mbox{ if }\ x,y\in S_m\ \mbox{for some }m\in\mb{N}\\
d_E(x,y) &\mbox{ otherwise }
\end{array}
\right.
\end{eqnarray*}
\end{defn}

\indent
Suppose that $x\in\mb{R}^s$ and $r>0$. We write $B_r(x)$ to denote the
Euclidean open ball with center $x$ and radius $r$, and $B^\phi_r(x)$ to
denote the open ball in $(\mb{R}^s,d^\phi)$.
Now we show that $(\mb{R}^s,d^\phi)$ is totally bounded.
\begin{lem}
$(\mb{R}^s,d^\phi)$ is totally bounded.
\end{lem}
\begin{proof}
Let $\epsilon>0$. We may assume that $\epsilon<1$. Choose $k\in\mathbf{N}$
such that
\begin{equation}\label{eq;tb1}
\frac{1}{1+k}<\frac{\epsilon}{4}\quad\mbox{and}\quad
\frac{1}{1+a_k}<\frac{\epsilon}{4},
\end{equation}
and let
$$B_{k+1}=\left\{x\in\mathbf{R}^s\mid d_E(O,x)\leq a_{k+1}\right\}.$$
Since $B_{k+1}$ is compact in $(\mathbf{R}^s,d_E)$, so is in $(\mathbf{R}^s,
d^\phi)$ by Corollary \ref{cor;homeo}.
Therefore we can cover $B_{k+1}$ with finite number of
$\epsilon$-balls in $(\mathbf{R}^s,d^\phi)$. Notice that $S_k\subset B_{k+1}$.
Since $S_k$ is also compact in $(\mathbf{R}^s,d_E)$, we can cover
$S_k$ with finite number of Euclidean $\frac{\epsilon}{4}$-balls with centers
$x_1,x_2,\cdots,x_N\in S_k$.
From eq. (\ref{eq;trivial}), we have
$$S_k\subset\bigcup_{i=1}^N B_{\frac{\epsilon}{4}}(x_i)
\subset\bigcup_{i=1}^N B_{\epsilon}(x_i)
\subset\bigcup_{i=1}^N B^\phi_{\epsilon}(x_i).$$
Note that if $z\in S_k$ then there exists $x_i\in\{x_1,x_2,\cdots,x_N\}
\subset S_k$ such that
$$d_E(z,x_i)<\frac{\epsilon}{4}.$$

\indent
To show that $(\mb{R}^s,d^\phi)$ is totally bounded, it is enough to show that
if $x\notin B_{k+1}$ then there exists $x_i\in\{x_1,x_2,\cdots,x_N\}$
such that $d^\phi(x,x_i)<\epsilon$.
Suppose that $x\notin B_{k+1}$. There exists $m\in\mathbf{N}$ such that
$$a_m\leq d_E(O,x)<a_{m+1}.$$
Since $x\notin B_{k+1}$, we have $k<m$. Let
$$y=\frac{a_m}{d_E(O,x)}\, x\in S_m.$$
From eq. (\ref{eq;tb1}), we have
\begin{equation}\label{eq;tb2}
d_E(x,y)<\frac{1}{1+m}<\frac{1}{1+k}<\frac{\epsilon}{4}.
\end{equation}
Let $z$ be the point in $S_k$ such that $h_{k,m}(z)=y$.
Choose $x_i\in\{x_1,x_2,\cdots,x_N\}$ such that
\begin{equation}\label{eq;tb3}
d_E(z,x_i)<\frac{\epsilon}{4}.
\end{equation}
From eq. (\ref{eq;trivial}), (\ref{eq;tb1}), (\ref{eq;tb2})
and (\ref{eq;tb3}), we have
\begin{eqnarray*}
d^\phi(x,x_i)&\leq& d^\phi(x,y)+d^\phi(y,z)+d^\phi(z,x_i)\\
&\leq& d_E(x,y)+\delta^\phi(y,z)+d_E(z,x_i)\\
&<& \frac{\epsilon}{4}+\frac{1}{1+a_m}+\frac{1}{1+a_k}+\frac{\epsilon}{4}\\
&<&\epsilon.
\end{eqnarray*}
\end{proof}

\indent
Since $(\mb{R}^s,d^\phi)$ is totally bounded, its completion
$(\overline{\mb{R}^s},\rho)=(\overline{\mb{R}^s},\rho_\phi)$ is a 
compactification of  $(\mathbf{R}^s,d_E)$, where we wrote simply $\rho$
to denote $\rho_\phi$ for the sake of simplicity. 
Recall that an element of $(\overline{\mb{R}^s},\rho)$ is an equivalence class
of Cauchy sequence in $(\mb{R}^s,d^\phi)$, where two Cauchy sequences $\{x_i\}$
and $\{y_i\}$ are equivalent if and only if
$$\lim_{i\to\infty}d^\phi(x_i,y_i)=0.$$

\noindent
Notice that if $\{x_i\}$ is a Cauchy sequence in $(\mb{R}^s,d^\phi)$ which
converges to $x$, then $\{x_i\}$ and the constant Cauchy sequence $\{x\}$ are
equivalent. Notice also that if $\{y_i\}$ is a subsequence of a Cauchy sequence
$\{x_i\}$, then they are equivalent.

\indent
Since for all $x\in S_1$, we have
$$d^\phi(a_ix,a_jx)\leq \delta^\phi(a_ix,a_jx)
\leq\frac{1}{1+a_i}+\frac{1}{1+a_j},$$
it is clear that $\{a_ix\}$ is a Cauchy sequence in $(\mb{R}^s,d^\phi)$.
By Lemma \ref{lem;useful}, we can show that $\{a_ix\}$ is not equivalent to
any constant Cauchy sequence (see the proof of Lemma \ref{lem;1-1}).
Furthermore, we have
\begin{lem}\label{lem;onto}
If $\{x_i\}$ is a Cauchy sequence in $(\mb{R}^s,d^\phi)$ which is not
equivalent to a constant Cauchy sequence, then it is equivalent to
$\{a_ix\}$ for some $x\in S_1$.
\end{lem}

\begin{proof}
Suppose that $\{x_i\}$ is a Cauchy sequence in $(\mb{R}^s,d^\phi)$ which is
not equivalent to a constant Cauchy sequence.
If $\{x_i\}$ is bounded in  $(\mb{R}^s,d_E)$, then it has a convergent
subsequence $\{y_i\}$, which converges to a point $y$ in $(\mb{R}^s,d_E)$.
Notice that $\{y_i\}$ converges to $y$ in $(\mb{R}^s,d^\phi)$, too.
Therefore $\{x_i\}$ is equivalent to $\{y_i\}$, and hence to the constant
Cauchy sequence $\{y\}$. This is a contradiction.

\indent
Since $\{x_i\}$ is unbounded in $(\mb{R}^s,d_E)$, we can choose a subsequence
of $x_i$, which we will call $x_i$ again, such that
$$0<d_E(O,x_i)<d_E(O, x_{i+1})\quad\mbox{for all }i\in\mb{N}$$
and there exists at most one $x_i$ such that
$$a_m\leq d_E(O,x_i)<a_{m+1}$$
for each $m\in\mb{N}$. Notice that $m\to\infty$ as $i\to\infty$. Since
$$\frac{1}{d_E(O,x_i)}\, x_i\in S_1$$
for all $i\in\mb{N}$ and $(S_1,d_E)$ is compact, $x_i$ has a subsequence,
which we will call $x_i$ again, such that
$$\frac{x_i}{d_E(O,x_i)}\ \mbox{converges to }x
\mbox{ for some }x\in S_1.$$

\indent
Suppose that $a_m\leq d_E(O,x_i)<a_{m+1}$. Let $y_i=a_m x$.
Notice that $\{y_i\}$ is a subsequence of $\{a_i x\}$. Let
$$z_i=\frac{a_m}{d_E(O,x_i)}x_i.$$
Since $d_E(x_i,z_i)\leq\frac{1}{m+1}$, we have
\begin{eqnarray*}
&&\lim_{i\to\infty}d^\phi(x_i,y_i)\\
&\leq&\lim_{i\to\infty}\left(d^\phi(x_i,z_i)+d^\phi(z_i,y_i)\right)\\
&\leq&\lim_{i\to\infty}\left(d_E(x_i,z_i)+\delta^\phi(z_i,y_i)\right)\\
&\leq& \lim_{i\to\infty}\left(\frac{1}{1+m}+
\frac{1}{1+a_m}+d_E\left(\frac{x_i}{d_E(O,x_i)},x\right)
+\frac{1}{1+a_m}\right)\\
&=&0.
\end{eqnarray*}
Therefore $\{x_i\}$ and $\{y_i\}$ are equivalent, and thus
$\{x_i\}$ is equivalent to $\{a_ix\}$.
\end{proof}

\indent
To show that $(\overline{\mb{R}^s},\rho)$ is homeomorphic to $(B^s,d_E)$,
we define a function
$$h:(B^s,d_E)\to (\overline{\mb{R}^s},\rho)$$ as follows.

\begin{eqnarray*}
h(x)=\left\{
\begin{array}{cl}
\frac{1}{1-d_E(O,x)}\, x \mbox{ (the constant Cauchy sequence)}
&\mbox{ if }d_E(O,x)<1\\
\{a_i x\} &\mbox{ if }d_E(O,x)=1
\end{array}
\right.
\end{eqnarray*}

\indent
Notice that
$$h\left(\frac{1}{1+d_E(O,y)}\, y \right)=y$$
for all $y\in\mb{R}^s$. Therefore from Lemma \ref{lem;onto}, it is clear that
$h$ is surjective. We will need the following lemma to show that
$h$ is injective.
\begin{lem}\label{lem;t1}
Suppose that $d_E(O,x)\geq 1$ and $d_E(O,y)\geq 1$.
Let $(x_0,x_1,\cdots,x_m)\in\Gamma_{x,y}$ with $d_E(O,x_i)<1$ for all
$1\leq i\leq m-1$. Then
$$\sum^m_{i=1} \delta^\phi(x_{i-1},x_i)\geq d_E\left(\frac{x}{d_E(O,x)},
\frac{y}{d_E(O,y)}\right).$$
\end{lem}

\begin{proof}
Notice that we may assume
$$\frac{x}{d_E(O,x)}\neq\frac{y}{d_E(O,y)}.$$
If $m=1$ then
\begin{eqnarray*}
&&\sum^m_{i=1} \delta^\phi(x_{i-1},x_i)\\
&=&\delta^\phi(x,y)\\
&=& \min\left\{d_E(x,y), \frac{1}{1+d_E(O,x)}+\phi(x,y)
+\frac{1}{1+d_E(O,y)}\right\}\\
&\geq& \min\left\{d_E(x,y), \frac{1}{1+d_E(O,x)}+ d_E\left(\frac{x}{d_E(O,x)},
\frac{y}{d_E(O,y)}\right)+\frac{1}{1+d_E(O,y)}\right\}\\
&\geq&d_E\left(\frac{x}{d_E(O,x)},\frac{y}{d_E(O,y)}\right).
\end{eqnarray*}

\indent
Suppose that $m\neq 1$. Notice that
$$\delta^\phi(x_{i-1},x_i)=d_E(x_{i-1},x_i)\quad\mbox{for all }1\leq i\leq m$$
and therefore
$$\sum^m_{i=1}\delta^\phi(x_{i-1},x_i)\geq \sum^m_{i=1}d_E(x_{i-1},x_i)
\geq d_E(x,y)
\geq d_E\left(\frac{x}{d_E(O,x)},\frac{y}{d_E(O,y)}\right).$$
\end{proof}

\indent
Now we show that $h$ is injective.
\begin{lem}\label{lem;1-1}
$h$ is injective.
\end{lem}

\begin{proof}
Suppose that $h(x)=h(y)$. We will show that $x=y$.
If $d_E(O,x)<1$ and $d_E(O,y)<1$, then
\begin{equation}\label{eq;xy}
\frac{1}{1-d_E(O,x)}\, x=\frac{1}{1-d_E(O,y)}\, y
\end{equation}
and therefore
$$\frac{1}{1-d_E(O,x)}\,d_E(O,x)=\frac{1}{1-d_E(O,y)}\, d_E(O,y).$$
Hence $d_E(O,x)=d_E(O,y)$. Thus from eq. (\ref{eq;xy}), we have $x=y$.

\indent
If $d_E(O,x)=1$ and $d_E(O,y)=1$, then the Cauchy sequences $\{a_i x\}$ and
$\{a_i y\}$ are equivalent.
Suppose that $x\neq y$. We will get a contradiction. Let
$$(x_0,x_1,\cdots,x_m)\in\Gamma_{a_ix,a_iy}.$$
Using Lemma \ref{lem;t1}, we can show that
$$\sum^m_{i=1}\delta^{\phi}(x_{i-1},x_i)\geq d_E(x,y).$$
and therefore
\begin{equation}\label{eq;g1}
d^\phi(a_i x, a_i y)\geq d_E(x,y)>0\quad\mbox{for all }i.
\end{equation}
Hence $\lim_{i\to\infty}d^\phi(a_i x, a_i y)\neq 0$.
This is a contradiction.

\indent
Suppose that $d_E(O,x)<1$, $d_E(O,y)=1$ and 
$$\lim_{i\to\infty}d^\phi\left(\frac{1}{1-d_E(O,x)}\, x,\,
a_i y\right)=0.$$
We will get a contradiction. Notice that if $i$ is large enough, then
$$d^\phi\left(\frac{1}{1-d_E(O,x)}\, x,\,
a_i y\right)\neq d_E\left(\frac{1}{1-d_E(O,x)}\, x,\,a_i y\right).$$
Therefore by Lemma \ref{lem;useful}, for large enough $i$, we have
$$d^\phi\left(\frac{1}{1-d_E(O,x)}\, x,\,a_i y\right)\geq
\frac{1}{2\left(1+d_E\left(O,\frac{1}{1-d_E(O,x)}x\right)\right)}>0.$$
Hence
$$\lim_{i\to\infty}d^\phi\left(\frac{1}{1-d_E(O,x)}\, x,\,
a_i y\right)\neq0.$$
This is a contradiction.
\end{proof}

\indent
Since $h$ is bijective, we can consider its inverse function.
Recall Lemma \ref{lem;onto} and let
$$k:(\overline{\mb{R}^s},\rho)\to (B^s,d_E)$$ be the function defined by
\begin{eqnarray*}
k(\{x_i\})=\left\{
\begin{array}{cl}
\frac{1}{1+d_E(O,x)}\, x
&\mbox{ if }\{x_i\}=\{x\} \mbox{ is a constant Cauchy sequence} \\
x &\mbox{ if }x_i=a_i x\mbox{ for some }x\in S_1.
\end{array}
\right.
\end{eqnarray*}

It is easy to show that $k$ is the inverse function of $h$.
In the following two lemmas, we will show that $h$ and $k$ are continuous.
Therefore $(\overline{\mb{R}^s},\rho)$ is homeomorphic to $(B^s,d_E)$.

\begin{lem}\label{lem;conti}
$h$ is continuous.
\end{lem}

\begin{proof}
Suppose that $x_n\to x$ in $(B^s,d_E)$. We will show that $h(x_n)\to h(x)$ in 
$(\overline{\mb{R}^s},\rho)$.
If $d_E(O,x)<1$, then it is trivial to show that $h(x_n)\to h(x)$ in
$(\mb{R}^s,d_E)$. Therefore from eq. (\ref{eq;trivial}), we have
$h(x_n)\to h(x)$ in $(\mb{R}^s,d^\phi)$, and hence in
$(\overline{\mb{R}^s},\rho)$.

\indent
Suppose that $d_E(O,x)=1$. Notice that it is enough to consider only the
following two cases,
\begin{enumerate}
\item[(a)]
$d_E(O,x_n)=1$ for all $n$
\item[(b)]
$d_E(O,x_n)<1$ for all $n$.
\end{enumerate}

\indent
For the case (a), we have
\begin{eqnarray*}
\rho(h(x_n),h(x))&=&\lim_{i\to\infty}d^\phi(a_i x_n, a_i x)\\
&\leq& \lim_{i\to\infty}\left(\frac{1}{1+a_i}+d_E(x_n,x)+\frac{1}{1+a_i}\right)
\\&=&d_E(x_n,x).
\end{eqnarray*}
Therefore if $x_n\to x$ in $(B^s,d_E)$, then $h(x_n)\to h(x)$ in
$(\overline{\mb{R}^s},\rho)$.

\indent
For the case (b), if
$$a_m\leq d_E(O,h(x_n))=d_E\left(O,\frac{1}{1-d_E(O,x_n)}\, x_n\right)
<a_{m+1},$$
let $$z_n=\frac{a_m}{d_E(O,h(x_n))}\, h(x_n)=\frac{a_m}{d_E(O,x_n)}\, x_n.$$
Notice that $z_n\in S_m$, and $m\to\infty$ as $n\to\infty$.
Therefore from eq. (\ref{eq;trivial}), we have
\begin{eqnarray*}
&&\lim_{n\to\infty}\rho(h(x_n),h(x))\\
&=&\lim_{n\to\infty}\lim_{i\to\infty}
d^\phi(h(x_n), a_i x)\\
&\leq&\lim_{n\to\infty}\lim_{i\to\infty}
 \left(d^\phi(h(x_n),z_n)+d^\phi(z_n,a_m x)+d^\phi(a_m x, a_i x)\right)\\
&\leq& \lim_{n\to\infty}\lim_{i\to\infty}
\left(d_E(h(x_n),z_n)+\delta^\phi(z_n,a_m x)+\delta^\phi(a_m x,a_i x)\right)\\
&\leq& \lim_{n\to\infty}\lim_{i\to\infty}
\left(\frac{1}{1+m}+\frac{1}{1+a_m}+ d_E\left(\frac{h(x_n)}{d_E(O,h(x_n))},
x\right)+\frac{1}{1+a_m}\right. \\
&&\qquad\qquad\qquad\qquad\qquad\qquad\qquad\qquad \left.
+\frac{1}{1+a_m}+\frac{1}{1+a_i}\right)\\
&\leq& \lim_{n\to\infty}d_E\left(\frac{x_n}{d_E(O,x_n)},x\right)\\
&=&0.
\end{eqnarray*}
Therefore $h(x_n)\to h(x)$ as $n\to\infty$.
\end{proof}

\begin{lem}\label{lem;iconti}
$k$ is continuous.
\end{lem}

\begin{proof}
Suppose that $\mb{x}_n=\{x_{n,i}\}$ converges to $\mb{x}=\{x_i\}$ in
$(\overline{\mb{R}^s},\rho)$.
We will show that $k(\mb{x}_n)$ converges to $k(\mb{x})$ in $(\mb{B}^s,d_E)$.

\indent
Suppose that $\mb{x}$ is equivalent to a constant Cauchy sequence $\{x\}$ in
$(\mb{R}^s,d^\phi)$.
If $\mb{x}_n$ is equivalent to $\{a_ix_n\}$ with $x_n\in S_1$ for infinitely 
many n, then choose a subsequence of $\mb{x}_n$, which we will call $\mb{x}_n$
again, such that $\mb{x}_n=\{a_ix_n\}$ with $x_n\in S_1$. Notice that
there exists $I>0$, which does not depend on $n$, such that
$$d_E(a_ix_n,x)\geq\frac{1}{2(1+d_E(O,x))}\quad\mbox{for all }i>I.$$
Therefore by Lemma \ref{lem;useful}, we have
$$d^\phi(a_ix_n,x)\geq\frac{1}{2(1+d_E(O,x))}\quad\mbox{for all }i>I.$$
Hence $\mb{x}_n$ does not converges to $\mb{x}$ in
$(\overline{\mb{R}^s},\rho)$. This is a contradiction.
Therefore $\mb{x}_n=\{x_n\}$ is a constant Cauchy sequence in
$(\mb{R}^s,d^\phi)$ for large enough $n$. 
Since $x_n$ converges to $x$ in $(\mb{R}^s,d^\phi)$, by Corollary 
\ref{cor;homeo}, $x_n$ converges to $x$ in $(\mb{R}^s,d_E)$. 
Therefore
$$k(\mb{x}_n)=\frac{1}{1+d_E(O,x_n)}\, x_n\quad\mbox{converges to}
\quad  k(\mb{x})=\frac{1}{1+d_E(O,x)}\, x.$$

\indent
If $\mb{x}=\{x_i\}$ is not equivalent to a constant Cauchy sequence in
$(\mb{R}^s,d^\phi)$, then by Lemma \ref{lem;onto}, we may assume
$x_i=a_i x$ for some $x\in S_1$. Notice that we may consider only the
following two cases.
\begin{enumerate}
\item[(a)]
For all $n$, $x_{n,i}=a_i x_n$ for some $x_n\in S_1.$
\item[(b)]
For all $n$, $\mb{x}_n=\{x_n\}$ is a constant Cauchy sequence.
\end{enumerate}

\indent
For the case (a), from eq. (\ref{eq;g1}) we have
\begin{eqnarray*}
0&=&\lim_{n\to\infty}\rho(\mb{x}_n,\mb{x})\\
&=&\lim_{n\to\infty}\lim_{i\to\infty}d^\phi(a_ix_n,a_ix)\\
&\geq&\lim_{n\to\infty}d_E(x_n,x)\\
&=&\lim_{n\to\infty}d_E(k(\mb{x}_n),k(\mb{x})).
\end{eqnarray*}

\indent
For the case (b), suppose that
$$\lim_{n\to\infty}d_E\left(\frac{1}{1+d_E(O,x_n)}\, x_n,x\right)\neq 0.$$
We will get a contradiction.
Choose a subsequence $\{y_n\}$ of $\{x_n\}$ such that
$$\frac{1}{1+d_E(O,y_n)}\, y_n\to y\neq x\quad\mbox{in}\quad (B^s,d_E).$$
Since $h$ is continuous and injective, we have
$$y_n=h\left(\frac{1}{1+d_E(O,y_n)}\, y_n\right)\to h(y)\neq h(x)=\mb{x}
\quad\mbox{in}\quad (\overline{\mb{R}^s},\rho).$$
Therefore $\lim_{n\to\infty}\rho(\mb{y}_n,\mb{x})\neq 0$.
This is a contradiction.
\end{proof}

\section{A compactification of $(\mb{R}^s,d_E)$ which
is not equivalent to the standard compactification}

Two compactifications $Y_1$ and $Y_2$ of a topological space $X$ are
called equivalent if there exists a homeomorphism $h:Y_1\to Y_2$
such that $h(x)=x$ for all $x\in X$. Recall that $s\geq 2$. In this section, 
we construct a compactification of $(\mb{R}^s,d_E)$ which is homeomorphic to 
the closed unit ball $(B^s,d_E)$, but not equivalent to the standard
compactification $(\overline{\mb{R}^s},\rho_\phi)$ in Section
\ref{sec;scpt}. We define a nonnegative symmetric function
$\psi:\mb{R}^s\times\mb{R}^s\to\mb{R}$ as follows.
Choose $0<\delta<\frac{\pi}{4}$ and let
$$A^+=\{x\in S_1\mid \angle xO\mb{a}_1\leq\delta\},\quad
A^-=\{x\in S_1\mid \angle xO(-\mb{a}_1)\leq\delta\},$$
where $\mb{a}_1=(1,0,\cdots,0)$ and $-\mb{a}_1=(-1,0,\cdots,0)\in\mb{R}^s$.
For each $x\in S_1$, let
$$P_x=\{t\mb{a}_1+t'x\in\mb{R}^s\mid t,t'\in\mb{R}\}.$$

\indent
We define an infinite ray $L_x\subset P_x$ starting from $x$ as follows.
See Figure \ref{fig}, where 
$$\theta=\frac{\pi}{\pi-2\delta}(\angle xO\mb{a}_1-\delta).$$

\begin{figure}[ht]
\begin{center}
\psfrag{O}{$O$}
\psfrag{L}{$L_x$}
\psfrag{x}{$x$}
\psfrag{y}{$y$}
\psfrag{a}{$\mb{a}_1$}
\psfrag{s}{$\theta$}
\psfrag{d}{$\delta$}
\includegraphics[width=3.2in,height=2in]{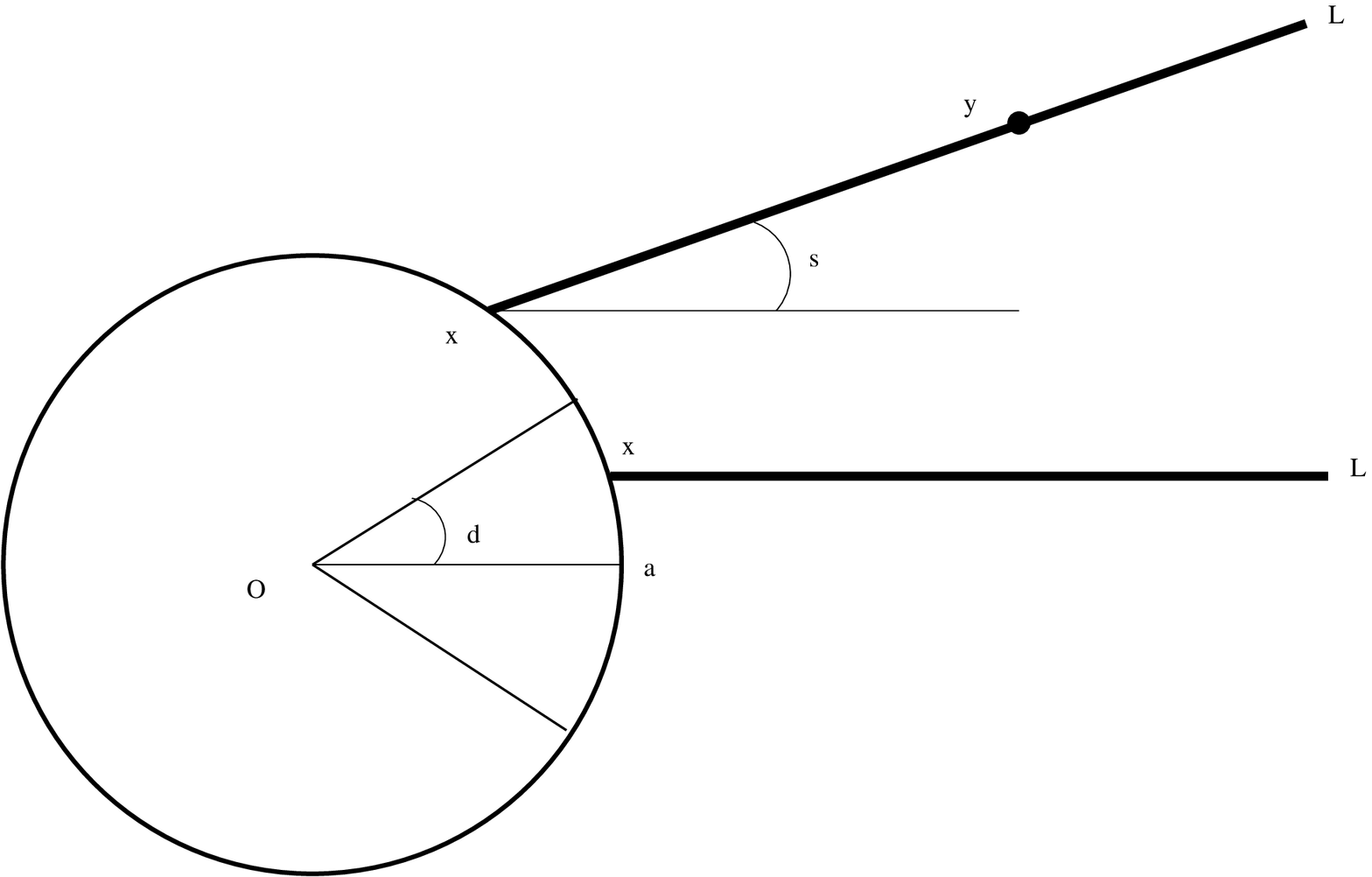}
\caption{$L_x$\label{fig}}
\end{center}
\end{figure}

\begin{eqnarray*}
L_x=\left\{
\begin{array}{cl}
\{x+t\mb{a}_1\mid t\geq 0\} &\mbox{ if }x\in A^+\\
\{x+t\mb{a}_1\mid t\leq 0\} &\mbox{ if }x\in A^-\\
\{x,y\in P_x\mid \angle (y-x)O\mb{a}_1=\frac{\pi}{\pi-2\delta}(\angle
xO\mb{a}_1-\delta)\} &\mbox{ if }x\in S_1\setminus (A^+\cup A^-)
\end{array}
\right.
\end{eqnarray*}
and let $L=\{L_x\mid x\in S_1\}$. Notice that 
\begin{enumerate}
\item[(i)]
If $\angle xO\mb{a}_1=\frac{\pi}{2}$, then 
$L_x=\{tx\mid t\geq1\}$.
\item[(ii)]
For all $x\in S_1$, the angle between two rays $L_x$ and $\{tx\mid t\geq1\}$
is not greater than $\delta$.
\item[(iii)]
For all $y\in \mb{R}^s$ with $d_E(O,y)\geq 1$, there exists 
unique ray in $L$ which is through $y$. 
\end{enumerate}

\indent
For all $p,q\in\mb{N}$, let 
$h_{p,q}:S_p\to S_q$ be the homeomorphism defined by
$$h_{p,q}(x)=\mbox{ the intersection of }S_q\mbox { and the ray in }L\mbox
{ which is through }x.$$
In particular, we have $h_{p,p}(x)=x$, and if $h_{p,q}(x)=y$ then 
$h_{q,p}(y)=x$. 
The nonnegative symmetric function $\psi$ is defined as follows.

\begin{defn}
\begin{eqnarray*}
\psi(x,y)=\left\{
\begin{array}{cl}
0 &\mbox{ if }\ h_{p,q}(x)=y\ \mbox{for some }p,q\in\mb{N}\\
d_E\left(h_{m,1}(x),h_{m,1}(y)\right) &\mbox{ if }\
x,y\in S_m \mbox{ for some }m\in\mb{N}\\
d_E(x,y) &\mbox{ otherwise.}
\end{array}
\right.
\end{eqnarray*}
\end{defn}

\indent
Similarly as in Section \ref{sec;scpt}, 
we can show that $(\mb{R}^s,d^\psi)=(\mb{R}^s,d_E^{\psi,O})$ is 
totally bounded, and its completion $(\overline{\mb{R}^s},\rho_\psi)$ is
homeomorphic to $(B^s,d_E)$ by the following homeomorphism
$h:(B^s,d_E)\to (\overline{\mb{R}^s},\rho_\psi)$,
\begin{eqnarray*}
h(x)=\left\{
\begin{array}{ll}
\frac{1}{1-d_E(O,x)}\, x & \mbox{ if }d_E(O,x)<\frac{1}{2}\\
y\in L_{\frac{x}{d_E(O,x)}}\mbox{ such that }&\\
\qquad
d_E\left(\frac{x}{d_E(O,x)},y\right)=\frac{d_E(O,x)-\frac{1}{2}}{1-d_E(O,x)} &
\mbox{ if }\frac{1}{2}\leq d_E(O,x)<1\\
\{h_{1,i}(x)\} &\mbox{ if }d_E(O,x)=1.
\end{array}
\right.
\end{eqnarray*}

\indent 
Suppose that $A,B\subset\mb{R}^s$. Let
$$d_E(A,B)=\inf\{d_E(x,y)\mid x\in A,\ y\in B\}.$$
In spherical coordinate system the distance between $(\rho_1,\phi_1,\theta_1)$ 
and $(\rho_2,\phi_2,\theta_2)$ is
\begin{equation}\label{eq;sqrt}
\sqrt{\rho_1^2+\rho_2^2-2\rho_1\rho_2\{\sin\phi_1\sin\phi_2\cos(
\theta_1-\theta_2)+\cos\phi_1\cos\phi_2\}}.
\end{equation}
The following two Lemmas are useful to show that $h$ is a homeomorphism.
\begin{lem}\label{lem;gd}
Suppose that $x,y\in S_1$. Then
$$d_E(L_x,L_y)\geq\frac{1}{2\sqrt{2}}\, d_E(x,y).$$
\end{lem}
\begin{proof}

\begin{figure}[ht]
\begin{center}
\psfrag{a}{$\mb{a}_1$}
\psfrag{al}{$\alpha$}
\psfrag{b}{$\beta$}
\psfrag{O}{$O$}
\psfrag{L1}{$L_x$}
\psfrag{L2}{$L^*_x$}
\psfrag{L3}{$L_w$}
\psfrag{x}{$x$}
\psfrag{w}{$w$}
\psfrag{p}{$\phi_2$}
\includegraphics[width=2.3in,height=2.3in]{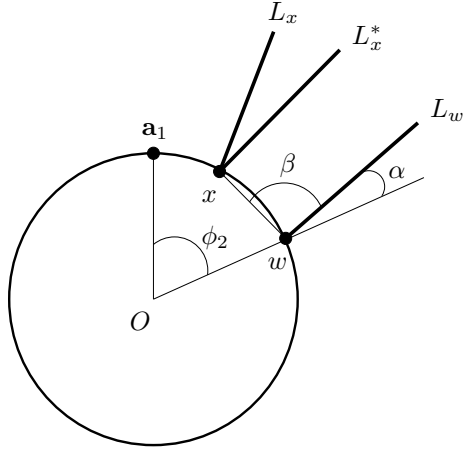}
\caption{Since $0\leq\alpha\leq\delta<\frac{\pi}{4}$, we have
$\frac{\pi}{4}<\beta\leq\frac{3\pi}{4}$.
\label{fig2}}
\end{center}
\end{figure}

\noindent
We may assume that $x\neq y$.
Since there exists a 3-dimensional subspace which contains $O$, $\mb{a}_1$, 
$-\mb{a}_1$, $x$ and $y$, we may assume that $\mb{R}^s=\mb{R}^3$. 
In spherical coordinates $(\rho,\phi,\theta)$, let 
$O=(0,0,0)$, $\mb{a}_1=(1,0,0)$, $-\mb{a}_1=(1,\pi,0)$, $x=(1,\phi_1,\theta_1)$
and $y=(1,\phi_2,\theta_2)$. By exchanging $x$ and $y$ if necessary, we may 
assume that 
$$0\leq\phi_1\leq\frac{\pi}{2}\quad\mbox{and}\quad\phi_1\leq\phi_2.$$

\indent
Suppose that $\phi_2\leq\frac{\pi}{2}$.
Let $z=(1,\phi_1,\theta_2)$ and $w=(1,\phi_2,\theta_1)$.
Since $d_E(x,y)\leq d_E(x,z)+d_E(z,y)$ and  $d_E(x,w)=d_E(z,y)$, we have
$$d_E(x,z)\geq\frac{1}{2}d_E(x,y)\quad\mbox{or}\quad
d_E(x,w)\geq\frac{1}{2}d_E(x,y).$$
If $d_E(x,z)\geq\frac{1}{2}d_E(x,y)$, let $P$ be the plane containing $x$ and 
$z$ which is perpendicular to $\mb{a}_1$.
Let $L'_x$ be the projection of $L_x$ to the plane $P$ and so is $L'_y$.  
Notice that we have 
$$d_E(L_x,L_y)\geq d_E(L'_x,L'_y)\geq d_E(x,z)\geq\frac{1}{2}\, d_E(x,y).$$
Suppose that $d_E(x,w)\geq\frac{1}{2}d_E(x,y)$. 
Let $L_x^*$ be the ray starting from $x$ with the same direction as $L_w$.
Since $L_w=\{(\rho,\phi,\theta_1)\mid (\rho,\phi,\theta_2)\in L_y\}$,
from eq. (\ref{eq;sqrt}) and Figure \ref{fig2}, we have
$$d_E(L_x,L_y)\geq d_E(L_x,L_w)\geq d_E(L_x^*,L_w)\geq \frac{1}{\sqrt{2}}\, 
d_E(x,w)\geq \frac{1}{2\sqrt{2}}\, d_E(x,y).$$

\begin{figure}[ht]
\begin{center}
\psfrag{a}{$\alpha$}
\psfrag{O}{$O$}
\psfrag{L1}{$L''_x$}
\psfrag{L2}{$L''_y$}
\psfrag{L3}{$L^{**}_x$}
\psfrag{L4}{$L^{**}_y$}
\psfrag{L5}{$L^*_{z'}$}
\psfrag{x}{$x$}
\psfrag{y}{$y$}
\psfrag{z}{$z'$}
\includegraphics[width=2.5in,height=3in]{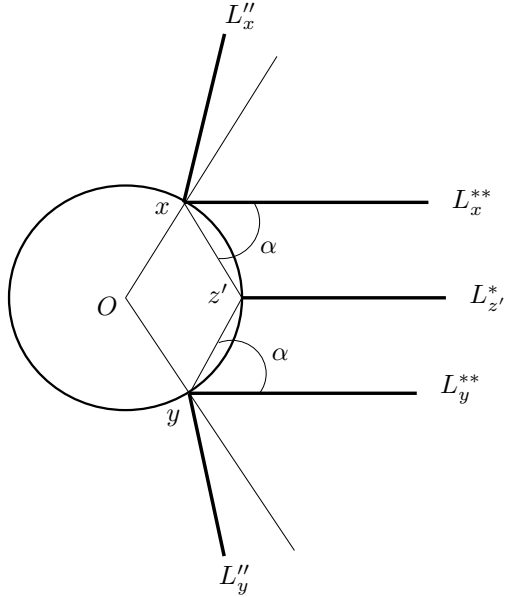}
\caption{$\frac{\pi}{4}\leq\alpha\leq\frac{\pi}{2}$\label{fig3}}
\end{center}
\end{figure}

\indent
Suppose that $\phi_2>\frac{\pi}{2}$.
Let $P'$ be the plane which contains the greatest circle in $S_1$  
through the points $x$ and $y$. 
Let $z'$ be the point on the greatest circle such that 
$$\angle xOz'=\angle yOz'\leq\frac{\pi}{2}.$$
Let $L''_x$ be the projection of $L_x$ to the plane $P'$ and so is $L''_y$. 
Let $$L^*_{z'}=\{tz'\mid t\geq 1\}.$$
Let $L^{**}_x$ be the ray from $x$ to the direction of $L^*_{z'}$ and so is 
$L^{**}_y$. From Figure \ref{fig3}, we have 
\begin{eqnarray*}
d_E(L_x,L_y)&\geq& d_E(L''_x,L''_y)\\
&\geq& d_E(L^{**},L^{**}_y)\\
&=& d_E(L^{**}_x,L^*_{z'})+d_E(L^*_{z'},L^{**}_y)\\
&\geq& \frac{1}{\sqrt{2}}\, d_E(x,z')+ \frac{1}{\sqrt{2}}\, d_E(z',y)\\
&\geq& \frac{1}{\sqrt{2}}\, d_E(x,y).
\end{eqnarray*}
\end{proof}

\begin{lem}
Suppose that $a_m\leq d_E(O,x)<a_{m+1}$.
Let $y$ be the intersection of $S_m$ and the ray in $L$ which is through $x$.
Then we have
$$d_E(y,x)\leq\frac{1}{\cos\delta}\, \frac{1}{m+1}\leq\frac{\sqrt{2}}{m+1}.$$
\end{lem}
\begin{proof}
Recall that $0<\delta<\frac{\pi}{4}$. From Figure \ref{fig4}, the proof is 
trivial.
\begin{figure}[ht]
\begin{center}
\psfrag{a}{$\alpha$}
\psfrag{b}{$\beta$}
\psfrag{O}{$O$}
\psfrag{L}{$L_z$}
\psfrag{x}{$x$}
\psfrag{y}{$y$}
\psfrag{z}{$z$}
\psfrag{S}{$S_1$}
\psfrag{S1}{$S_m$}
\psfrag{S2}{$S_{m+1}$}
\psfrag{d}{$\delta$}
\includegraphics[width=3.3in,height=2in]{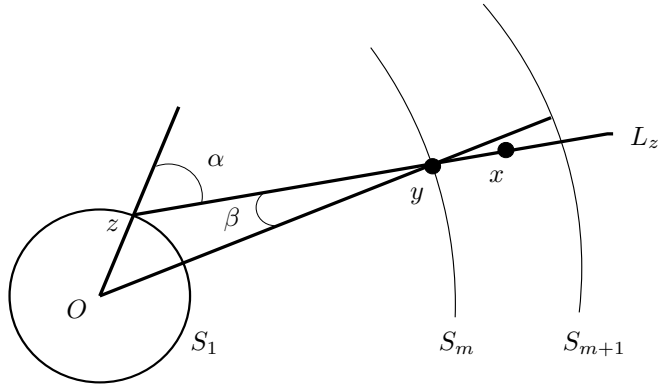}
\caption{$\beta<\alpha\leq\delta<\frac{\pi}{4}$\label{fig4}}
\end{center}
\end{figure}
\end{proof}

\indent
The following Lemma is also useful to show that $h$ is a homeomorphism.
Similarly as Lemma \ref{lem;t1}, we can prove this lemma.  
\begin{lem}\label{lem;t2}
Suppose that $d_E(O,x)\geq 1$ and $d_E(O,y)\geq 1$.
Suppose also that $x\in L_{x'}$ and $y\in L_{y'}$ with $x',y'\in S_1$.
Let $(x_0,x_1,\cdots,x_m)\in\Gamma_{x,y}$ with $d_E(O,x_i)<1$ for all
$1\leq i\leq m-1$. Then
$$\sum^m_{i=1} \delta^\psi(x_{i-1},x_i)\geq d_E\left(L_{x'},L_{y'}\right).$$
\end{lem}

\indent
Now we show that the compactification $(\overline{\mb{R}^s},\rho_\psi)$ of
$(\mb{R}^s,d_E)$ is not equivalent to the standard compactification
$(\overline{\mb{R}^s},\rho_\phi)$ in Section \ref{sec;scpt}.
\begin{prop}\label{prop;noteq}
$(\overline{\mb{R}^s},\rho_\psi)$ and $(\overline{\mb{R}^s},\rho_\phi)$ are not
equivalent compactifications.
\end{prop}

\begin{proof}
Suppose that they are equivalent. There exists a homeomorphism
$$h:(\overline{\mb{R}^s},\rho_\psi)\to (\overline{\mb{R}^s},\rho_\phi)$$
such that $h(x)=x$ for all $x\in\mb{R}^s$.
Choose a point $\mb{b}_1\in S_1$ such that 
$$\angle \mb{b}_1O\mb{a}_1=\frac{\delta}{2}.$$
Let $\mb{a}=\{\mb{a}_i\}$ and $\mb{b}=\{\mb{b}_i\}$, where 
$$\mb{a}_i=h_{1,i}(\mb{a}_1)=(a_i,0,\cdots,0)\quad\mbox{and}\quad 
\mb{b}_i=h_{1,i}(\mb{b}_1)$$ for all $i\in\mb{N}$. Notice that
$$\sin\frac{\delta}{2}\leq d_E(\mb{a}_i,\mb{b}_i)\leq \frac{\delta}{2}
\quad\mbox{for all }i.$$

\indent
Suppose that $(x_0,x_1,\cdots,x_m)\in\Gamma_{\mb{a}_i,\mb{b}_i}$.
Using Lemma \ref{lem;gd} and \ref{lem;t2}, we can show that
$$\sum^m_{i=1} \delta^\psi(x_{i-1},x_i)\geq
\frac{1}{2\sqrt{2}}d_E(\mb{a}_1,\mb{b}_1)\geq\frac{1}{2\sqrt{2}}
\sin\frac{\delta}{2}>0.$$  
Therefore
$$d^\psi(\mb{a}_i,\mb{b}_i)\geq\frac{1}{2\sqrt{2}}\sin\frac{\delta}{2}
\quad\mbox{for all }i,$$
and hence $\rho_\psi(\mb{a},\mb{b})\neq 0$. Thus $\mb{a}\neq\mb{b}$ in
$(\overline{\mb{R}^s},\rho_\psi)$.

\indent
But we have
\begin{eqnarray*}
\rho_\phi(h(\mb{a}),h(\mb{b}))
&=&\lim_{i\to\infty}\rho_\phi(h(\mb{a}_i),h(\mb{b}_i))\\
&=&\lim_{i\to\infty}\rho_\phi(\mb{a}_i,\mb{b}_i)\\
&=&\lim_{i\to\infty}d^\phi(\mb{a}_i,\mb{b}_i)\\
&\leq&\lim_{i\to\infty}\delta^\phi(\mb{a}_i,\mb{b}_i)\\
&\leq& \lim_{i\to\infty}\left(\frac{1}{1+a_i}+\frac{1}{a_i}
d_E(\mb{a}_i,\mb{b}_i)+\frac{1}{1+a_i}\right)\\
&\leq& \lim_{i\to\infty}\left(\frac{2}{1+a_i}+\frac{\delta}{2a_i}
\right)\\
&=& 0.
\end{eqnarray*}
Therefore $h(\mb{a})=h(\mb{b})$ in $(\overline{\mb{R}^s},\rho_\phi)$.
This is a contradiction.
\end{proof}

\end{document}